\documentclass[review,onefignum,onetabnum]{siamart171218}

\usepackage{lipsum}
\usepackage{amsfonts}
\usepackage{graphicx}
\usepackage{epstopdf}
\usepackage{algorithmic}
\ifpdf
  \DeclareGraphicsExtensions{.eps,.pdf,.png,.jpg}
\else
  \DeclareGraphicsExtensions{.eps}
\fi

\newsiamremark{remark}{Remark}
\newsiamremark{hypothesis}{Hypothesis}
\crefname{hypothesis}{Hypothesis}{Hypotheses}
\newsiamthm{claim}{Claim}

\headers{$\Gamma$-convergence of Nonlocal Dirichlet Energies}{Weiye Gan, Qiang Du, ZUOQIANG SHI}

\title{$\Gamma$-convergence of Nonlocal Dirichlet Energies With Penalty Formulations of Dirichlet Boundary Data
\thanks{This work was supported by National Natural Science Foundation of China under grant 12071244 and US National Science Foundation 
 DMS-2309245 and DMS-1937254.
}}

\author{Weiye Gan\thanks{Department of Mathematical Sciences, Tsinghua University, Beijing, 100084, China. Email:\textit{gwy22@mails.tsinghua.edu.cn}.}
\and Qiang Du\thanks{Department of Applied Physics and Applied Mathematics, and Data Science Institute, Columbia
University, New York, NY 10027, USA.  Email: \textit{qd2125@columbia.edu} }
\and Zuoqiang Shi\thanks{Yau Mathematical Sciences Center, Tsinghua University, Beijing, 100084, China \& Yanqi Lake Beijing Institute of Mathematical Sciences and Applications, Beijing, 101408, China. Email:\textit{zqshi@tsinghua.edu.cn} (Corresponding author)}}

\usepackage{amsopn}

\makeatletter
\newcommand*{\addFileDependency}[1]{
  \typeout{(#1)}
  \@addtofilelist{#1}
  \IfFileExists{#1}{}{\typeout{No file #1.}}
}
\makeatother

\newcommand{\gwy}[1]{{\color{black} #1}}
\newcommand{\GWY}[1]{{\color{black} #1}}
\newcommand{\qd}[1]{{\color{black} #1}}

\usepackage{mathtools}

\begin{document}
	\maketitle
	\begin{abstract}
We study nonlocal Dirichlet energies associated with a class of nonlocal diffusion models on a bounded domain subject to the conventional local Dirichlet boundary condition. {The goal of this paper is to give a general framework to correctly impose Dirichlet boundary condition in nonlocal diffusion model. To achieve this, we formulate the Dirichlet boundary condition as a penalty term and use theory of $\varGamma$-convergence to study the correct form of the penalty term. Based on the analysis of $\varGamma$-convergence, we prove that the Dirichlet boundary condition can be correctly imposed in nonlocal diffusion model in the sense of $\varGamma$-convergence as long as the penalty term satisfies a few mild conditions. This work provides a theoretical foundation for approximate Dirichlet boundary condition in nonlocal diffusion model.} 
	\end{abstract}


	\section{Introduction}
	\label{sec:introduction}
	Nonlocal models have been extensively 
studied in various scientific disciplines\cite{alfaro2017propagation,bavzant2002nonlocal,blandin2016well,dayal2007real,foss2022convergence,kao2010random, silling2000reformulation,vazquez2012nonlinear}. Among them, nonlocal models with operators that only consider nonlocal interactions of a limited range are of particular interest\cite{du2012analysis,du2019nonlocal
 }, as they are closely related to peridynamic\cite{silling2000reformulation} and some meshless numerical algorithms like smoothed particle hydrodynamics (SPH)\cite{gingold1977smoothed,lucy1977numerical,morris1997modeling}. In this paper, we concentrate on 
nonlocal energies associated with the nonlocal Laplacian 
that are the nonlocal counterpart of
 the local Laplacian. 
	While various theoretical and numerical studies have been devoted to problems associated with nonlocal Laplacians, or also called nonlocal diffusion models
\cite{d2020numerical,du2013nonlocal,mengesha2013analysis,trask2019asymptotically,zhang2018accurate}, establishing nonlocal analogues of the boundary conditions remains 
a topic of ongoing discussion. One approach is to extend the boundary to a small volume adjacent to the boundary, known as volume constraints\cite{du2012analysis}. Designing nonlocal models or volume constraints  properly can achieve better convergence rates to their local limit, such as for Neumann boundary condition in one\cite{tao2017nonlocal} and two dimensions\cite{you2020asymptotically}. \gwy{The point integral method\cite{li2017point,shi2017convergence} is also an effective approach to construct nonlocal 
approximations of the Poisson equation with Neumann boundary condition.} As for the Dirichlet boundary condition, the constant extension method\cite{macia2011theoretical} may be the most straightforward, but it only provides first-order convergence at best. Enforcing the no-slip condition for higher-order convergence may be costly due to the need to calculate particle distances from domain boundaries\cite{holmes2011smooth,yildiz2009sph}. Recently, 
there have been many studies devoted to constructing nonlocal models with respect to the Dirichlet condition
in different implementations\cite{d2022prescription,yang2022uniform,lee2021second,shi2017convergence,shi2018harmonic,zhang2021second,scott2023nonlocal}.
	
	In this paper, we mainly analyze a nonlocal
counterpart of the following well-known Dirichlet energy defined on a bounded domain $\Omega\subset \mathbb{R}^d$
 \begin{equation}
 \label{eq:p-diri}
\int_\Omega|\nabla u(x)|^pdx, \quad p\in (1,\infty)
	\end{equation}
 for functions $u\in W^{1,p}(\Omega)$, subject to Dirichlet boundary conditions on $\partial\Omega$.

 Given a nonlocal horizon parameter $\delta>0$ that controls the range of nonlocal interaction, a commonly studied nonlocal Direclet energy is given by
 	\begin{equation}
		\label{eq:nl-p-diri}
		\frac{1}{\delta^p}\int_\Omega \int_\Omega R_\delta(|x-y|)|u(x)-u(y)|^pdydx.
	\end{equation}
 where $R_\delta(s)=\frac{1}{\delta^d}R\left(\frac{s^2}{\delta^2}\right)$ is a scaled nonlocal kernel.

\gwy{It is well-known that variational problems associated with the minimization of the functional \eqref{eq:p-diri} for a Dirichlet data $a$ being a proper function prescribed on boundary $\partial\Omega$  lead to
the following homogeneous Poisson equation, also called \qd{the p-}Laplace equation or diffusion equation, with the Dirichlet boundary condition:}
	\gwy{\begin{equation}
		\label{eq:p-laplace}
  	\left\{\begin{array}{cc}
			\nabla\cdot(|\nabla u|^{\qd{p-2}}\nabla u)=0&\text{in }\Omega,\\
			u=a&\text{on }\partial\Omega	\end{array}\right.
	\end{equation}}
 For example, for $p=2$, we get the linear Dirichlet boundary value problem
 \begin{equation}
	\label{eq:laplace}
		\left\{\begin{array}{cc}
			\Delta u=0&\text{in }\Omega,\\
			u=a&\text{on }\partial\Omega	\end{array}\right.
	\end{equation}
For widely used kernels, particularly those smoothly defined $R_\delta$, 
functions with a finite nonlocal energy are not expected to have sufficient regularity to have well-defined traces on $\partial\Omega$, thus making it hard to directly impose the Dirichlet boundary condition like that in \eqref{eq:laplace}. A possible remedy is to adopt the technique of heterogeneous localization that leads to improved regularity at the boundary, see \cite{scott2023nonlocal} and the references cited therein.

There are other attempts on imposing Dirichlet boundary condition for the nonlocal diffusion model. \qd{For $p=2$, by} taking  a constant $0<\beta\ll 1$ and using the Robin boundary condition:$$u+\beta\frac{\partial u}{\partial \mathbf{n}}=a,$$
to approximate the Dirichlet boundary condition $u=a$,
a nonlocal model was proposed in \cite{shi2018harmonic} as follows:
\begin{equation*}
   \frac{4}{\delta^2}\int_\Omega R_\delta(|x-y|)(u(x)-u(y))d y-\frac{2}{\beta}\int_{\partial \Omega}\bar{R}_\delta(|x-y|)(a(y)-u(y))d S_y=0
\end{equation*}
where $\bar{R}_\delta(s)=\frac{1}{\delta^d}\bar{R}\left(\frac{s^2}{\delta^2}\right)$ and $\bar{R}(s)=\int_s^{+\infty}R(r)dr$. 
This nonlocal model is proved to converge to the local Laplace equation as $\delta,\beta\rightarrow 0$ \cite{shi2018harmonic} and error estimate in terms of $\delta$ and $\beta$ is also given.
However, the symmetry is destroyed in above nonlocal model such that it does not have variational form with nonlocal Dirichlet energy. By 
introducing $\frac{\partial u}{\partial \mathbf{n}}$ as an auxiliary variable, a nonlocal energy with penalty term can be derived \cite{wang2023nonlocal}.
\begin{align}
\label{eq:nonlocal-diffusion-1}
F(u,a)=&\frac{1}{\delta^2}\int_\Omega\int_\Omega R_\delta(|x-y|)(u(x)-u(y))^2d xd y \\
&+\int_{\partial \Omega} \frac{2}{\delta^2 \bar{\bar{\omega}}_\delta(x)}
\left(\int_\Omega 
\bar{R}_\delta(|x-y|)
(a(x)-u(y)) dy\right)^2dx\nonumber
\end{align}
where $$\bar{\bar{\omega}}_\delta(x)=\int_{\Omega}\bar{\bar{R}}_\delta(|x-y|)dy,$$ $\bar{\bar{R}}_\delta(s)=\frac{1}{\delta^d}\bar{\bar{R}}\left(\frac{s^2}{\delta^2}\right)$and $\bar{\bar{R}}(s)=\int_s^{+\infty}\bar{R}(r)dr$. 
It can be proved that minimal solution of above energy function converges to the solution of \eqref{eq:laplace} as $\delta\rightarrow 0$ and the convergence rate is $O(\delta)$ in $H^1$ norm \cite{wang2023nonlocal}. To get $H^1$ convergence, we need a specifically designed penalty term in \eqref{eq:nonlocal-diffusion-1}. \GWY{Previous researches primarily focus on carefully designing penalty terms to ensure the best possible convergence of nonlocal models. In contrast, in this paper, we are more concerned with identifying the conditions under which the convergence of the nonlocal model can be guaranteed. Since these conditions should be as weak as possible, we are committed to study a more general form of the penalty term.} 

\GWY{Motivated by the penalty formulation of the Dirichlet boundary value problems for ellptic PDEs and nonlocal energy \eqref{eq:nonlocal-diffusion-1}, we first consider the following nonlocal energy with a general penalty
term}\GWY{
\begin{equation}
\label{eq:penalty}
\begin{aligned}
& \frac{1}{\delta^p} \int_\Omega\int_\Omega R_{\delta}(|x-y|)|u(x)-u(y)|^pdxdy +B_p(u,a)
\\
&\qquad\text{where }\;
B_p(u,a) = \int_{\partial\Omega}\left|\frac{1}{\delta}\int_\Omega K_{\delta}(|x-y|)(a(x)-u(y))dy\right|^pdx,
\end{aligned}
	\end{equation}
 and $K_{\delta}$ is a nonlocal kernel that depends on the horizon parameter $\delta$. Our interests is to give correct form of $K_\delta$ such that above nonlocal model converges to the local model with Dirichlet boundary condition as $\delta\to 0$. Since we want to get a general form of $K_\delta$, $\varGamma$-convergence of above nonlocal model is analyzed rather than other strong convergence studied in previous works.}

 In fact, $\varGamma$-convergence among functionals is significant to describe the relationship between nonlocal operators and local ones in semi-supervised learning and other fields \cite{BBM01,slepcev2019analysis,roith2022continuum,garcia2016continuum,ponce2004new}.  With a property of $\varGamma$-convergence, we can also demonstrate the convergence of the minimizers (or solutions of the stationary equations).  Nevertheless, it does not provide any information about the convergence rate of the minimizers. \GWY{Hence, $\varGamma$-convergence is a weaker convergence and allow us to consider more relaxed conditions for the penalty term.}

For $p=2$, our $\varGamma$-convergence result covers the case of linear variational problems associated with the nonlocal energy \eqref{eq:nl-p-diri} with $p=2$ and the
 boundary penalty term of the form
 	\begin{equation}
		\label{eq:example}
		E(u,a)=\frac{1}{\delta^2}\int_{\partial\Omega}\left(\int_{\Omega}K_{\delta}(|x-y|)(u(y)-a(x))dy\right)^2dx
	\end{equation} 
 \GWY{that has been previously considered in \cite{wang2023nonlocal} for some special choices of $K_\delta$ connected to $R_\delta$ (see more details in \cref{remark1}). In this paper, we work with more general choices of $K_\delta$ that enables the $\varGamma$-convergence analysis, which remains valid for more general $p$. The kernel $K_\delta$ only need to fulfill some regularization conditions and do not need to have any connection with the kernel $R_\delta$.} Moreover, as another contribution, 
 we also discuss the convergence of associated  nonlocal eigenvalue problems.   In addition, specializing to linear problems corresponding to $p=2$, we consider a  penalty form  more general than
 that in \eqref{eq:example} to illustrate the broad applicability of the method developed here.

 \GWY{Based on the analysis of $\varGamma$-convergence, we establish a class of penalty models which are guaranteed to be correct approximation of the local Dirichlet boundary condition. This is also the main contribution of this paper.}
	
	The rest of the paper is organized as follows. In \cref{sec: assump and result}, we state the main results and all assumptions that
 we need. \gwy{Some related work considering special cases of our results is discussed in \cref{sec:quant-discussion}}. In \cref{sec:preliminaries}, $\varGamma$-convergence and some estimations employed in our proof procedure are introduced. $\varGamma$-convergence of the nonlocal models and compactness results are demonstrated in \cref{sec:proof} and \cref{sec:compact} respectively. In \cref{sec:conclusion}, we conclude this paper and present several aspects for future research.
	
	\section{Assumptions and main results}
	\label{sec: assump and result}
	Let $p>1$ be a finite constant. Suppose that $\Omega$ is a Lipschitz bounded domain in $\mathbb{R}^d$. $K,R$ are two kernel functions satisfying the following regularity conditions, 
	\begin{enumerate}
		\item[(K1)] $K,R:[0,\infty)\longrightarrow[0,\infty)$ belong to $C^1$,
		\gwy{\item[(K2)] $K,R$ are monotonically decreasing,}
		\item[(K3)] $ \mathrm{supp}(K)\subset[0,r_K^2
  ]$ and $\mathrm{supp}(R)\subset[0,r_R^2]$ 
  for some $r_K,r_R>0$.
	\end{enumerate}
 \qd{For the kernel $R$, we define a normalization constant 
  \begin{equation}\label{eq:sigmaR}
      \sigma_R\coloneqq\int_{\mathbb{R}^d}R(|z|^2)|z\cdot e_1|^pdz
  \end{equation}
that  only depends on a kernel $R$ with $e_1\coloneqq(1,0,\dots,0)$. } 
	For an arbitrary positive constant $\delta$, we also employ the scaled kernels $K_\delta(s)\coloneqq \frac{1}{\delta^d}K(\frac{s^2}{\delta^2})$ and $R_\delta(s)\coloneqq \frac{1}{\delta^d}R(\frac{s^2}{\delta^2})$. We consider the p-Laplace equation with Dirichlet boundary condition \eqref{eq:p-laplace} where $a$ is the trace of some function in $W^{1,p}(\Omega)$. 
 Then the weak solution of \eqref{eq:p-laplace} is also the minimizer of the following functional:
	\gwy{\begin{equation}
		\label{eq:con-fun}
		F(u)=\left\{\begin{array}{cc}
			\int_\Omega|\nabla u(x)|^pdx&\text{if }u\in W^{1,p},Tu=a\text{ on }\partial\Omega,\\
			\infty&\text{otherwise,}
		\end{array}\right.
	\end{equation}
  where $T$ is the trace operator for $W^{1,p}$ function. }
 	We firstly establish a specific nonlocal model for \eqref{eq:con-fun}: 
	\begin{equation}
		\label{eq:dis-fun}
		\begin{aligned}
			F_n(u)&=\frac{1}{\delta_n^p}\int_\Omega\int_\Omega R_{\delta_n}(|x-y|)|u(x)-u(y)|^pdxdy\\
			&+\int_{\partial\Omega}\left|\frac{1}{\delta_n}\int_\Omega K_{\delta_n}(|x-y|)(a(x)-u(y))dy\right|^pdx,
		\end{aligned}
	\end{equation}
 \gwy{and prove the $\varGamma$-convergence from $F_n(u)$ to $\sigma_RF(u)$ as the following theorem:}
	\gwy{\begin{theorem}\label{thm:gamma-con}
		Suppose that $\Omega$ is a Lipschitz bounded domain in $\mathbb{R}^d$. $1<p<\infty$ is a constant. $K,R$ are two kernel functions satisfying (K1)-(K3), \qd{with $\sigma_R$ given by \eqref{eq:sigmaR}}. $\{\delta_n\}$ is a sequence of positive constants tending to 0 as $n\rightarrow\infty$. Then we have
		\[F_n\stackrel{\varGamma}{\longrightarrow}\sigma_RF\quad\text{in }L^p(\Omega),\]
		where $F_n,F$ are defined as \eqref{eq:dis-fun},\eqref{eq:con-fun}.
	\end{theorem}}
 
	\gwy{Subsequently, We consider the eigenfunctions  of Laplace operator with Dirichlet condition \eqref{eq:eigenvalue}. In this part we set $p=2$.}
	\begin{equation}
		\label{eq:eigenvalue}
		\left\{\begin{array}{cc}
			\gwy{-\Delta u=\lambda u}&\text{in }\Omega,\\
			u=0&\text{on }\partial\Omega,\\
			\lVert u \rVert_{L^2(\Omega)}=1.
		\end{array}\right.
	\end{equation}
	From the standard theory of second-order elliptic equations(for example, Section 6.5 in \cite{Evans}), if we denote $\Sigma$ the set of all eigenvalues of Laplace operator $\Delta$, then $\Sigma=\{\lambda_k\}_{k=1}^\infty$ where
	\gwy{\[0<\lambda_1<\lambda_2\leq\dots\leq\lambda_k\leq\dots\]}
	Moreover,
	\[\lambda_1=\min\{\lVert\nabla u\rVert_{L^2(\Omega)}\ \big|\ u\in H_0^1(\Omega),\lVert u \rVert_{L^2(\Omega)}=1\}\]
	\gwy{and the corresponding eigenfunction $\mu_1$ (which means $(\mu_1,\lambda_1)$ is the solution of \eqref{eq:eigenvalue}) is the minimizer of the following functional:}
	\begin{equation*}
		F_e^1(u)=\left\{\begin{array}{cc}
			\int_\Omega|\nabla u(x)|^2dx&\text{if }u\in H_0^1,\lVert u \rVert_{L^2(\Omega)}=1,\\
			\infty&\text{otherwise,}
		\end{array}\right.
	\end{equation*}
	For $k\geq 1$, let $V_k=span\{\mu_1,\mu_2,\dots,\mu_k\}$. Then,
	\[\lambda_{k+1}=\min\{\lVert\nabla u\rVert_{L^2(\Omega)}\ \big|\ u\in H_0^1(\Omega),\lVert u \rVert_{L^2(\Omega)}=1,u\perp V_k\}\]
	and $\mu_k$ is the minimizer of
	\begin{equation}
		\label{eq:eigen-cons-func}
		F_e^{k+1}(u)=\left\{\begin{array}{cc}
			\int_\Omega|\nabla u(x)|^2dx&\text{if }u\in H_0^1,\lVert u \rVert_{L^2(\Omega)}=1,u\perp V_k\\
			\infty&\text{otherwise.}
		\end{array}\right.
	\end{equation}
    For this problem, selecting $a\equiv 0$, $p=2$, then the local functional \eqref{eq:con-fun} becomes
		\begin{equation}
		\label{eq:con-fun-special}
F(u)=\left\{\begin{array}{cc}
			\int_\Omega|\nabla u(x)|^2dx&\text{if }u\in H_0^1,\\
			\infty&\text{otherwise.}
		\end{array}\right.
	\end{equation}
	Note that \eqref{eq:eigen-cons-func} is actually \eqref{eq:con-fun-special} with additional constraints. We can construct the nonlocal approximation of $F_e^k$ by defining the unnormalized functionals:
	\begin{equation}
		\label{eq:eigen-dis-fun}
		\begin{aligned}
			F_{e,n}^{k}(u)&=\left\{\begin{array}{cc}
				F_n(u)&\text{if }\lVert u \rVert_{L^2(\Omega)}=1,u\perp V_{k-1}^n,\\
				\infty&\text{otherwise}
			\end{array}\right.
		\end{aligned}
	\end{equation}
        where $V_0=\emptyset,\ V_k^n=span\{\mu_1^n,\dots\mu_k^n\}$ and $\mu_k^n$ is the minimizer of $F_{e,n}^k.$
	Moreover, with a kernel function $W$ satisfying conditions (K1)-(K3) and \begin{equation}
\label{eq:normalW}
\int_{\mathbb{R}^d} W(|z|^2)dz=1,	\end{equation}
    we can consider the inner product defined as
    \[\langle u,v\rangle_n=\int_\Omega \int_\Omega W_{\delta_n} (|x-y|) u(x) v(y)dx dy\]
    and the normalized functionals defined as
	\begin{equation}
		\label{eq:normal-dis-fun}
		\begin{aligned}
			\tilde{F}_{e,n}^{k}(u)&=\left\{\begin{array}{cc}
				F_n(u)&\text{if }
    \langle u,u\rangle_n
    =1,\,
    u\perp \tilde{V}_{k-1}^n,\\
				\infty&\text{otherwise,}
			\end{array}\right.
		\end{aligned}
	\end{equation}
	where \gwy{$W_{\delta}(s)\coloneqq\frac{1}{\delta^d}W(\frac{s^2}{\delta^2})$}, $\tilde{V}_0=\emptyset,\ \tilde{V}_k^n=span\{\tilde{\mu}_1^n,\dots\tilde{\mu}_k^n\}$ and $\tilde{\mu}_k^n$ is the minimizer of $\tilde{F}_{e,n}^k.$
 We also have the $\varGamma$-convergence from $F_{e,n}^k$ or $\tilde{F}_{e,n}^k$ to $F_e^k$:
	\gwy{\begin{theorem}
		\label{thm:eigenvalue}
		Suppose that $\Omega$ is a Lipschitz bounded domain in $\mathbb{R}^d$.  $K,R,W$ are three kernel functions satisfying (K1)-(K3), \qd{with $W$ satisfying} additionally \eqref{eq:normalW},  \qd{and $\sigma_R$ being given by \eqref{eq:sigmaR}}. $\{\delta_n\}$ is a sequence of positive constants tending to 0 as $n\rightarrow\infty$. Then we have
		\[F_{e,n}^k\stackrel{\varGamma}{\longrightarrow}\gwy{\sigma_RF_e^k}\quad\text{in }L^2(\Omega),\]
		and
		\[\tilde{F}_{e,n}^k\stackrel{\varGamma}{\longrightarrow}\gwy{\sigma_RF_e^k}\quad\text{in }L^2(\Omega),\]
		for all $k\in\mathbb{N}$, where $F_{e,n}^k,\tilde{F}_{e,n}^k$ and $F_e^k$ are defined as \eqref{eq:eigen-dis-fun}, \eqref{eq:normal-dis-fun} and \eqref{eq:eigen-cons-func}.
	\end{theorem}}
	
	Proceeding to focus on the case where $a\equiv 0$ and $p=2$, we can consider a more general nonlocal functional $F_n$:
	\begin{equation}
		\label{eq:dis-fun-general}
		\begin{aligned}
			F_n(u)&=\frac{1}{\delta_n^2}\int_\Omega\int_\Omega R_{\delta_n}(|x-y|)|u(x)-u(y)|^2dxdy
			+E_n(u,0),
		\end{aligned}
	\end{equation}
where the more general formulation of the boundary term is given by
	\begin{equation}
		\label{eq:boudary-term}
		\begin{aligned}
			E_n(u,a)=\int_{\partial\Omega}\int_\Omega\int_\Omega\rho_{\delta_n,x}(y,z)(u(y)-a(x))(u(z)-a(x))dydzdx
		\end{aligned}
	\end{equation}
	with a kernel $\rho_{\delta_n,x}(y,z)$ 
 symmetric with respect to $y$ and $z$. \cref{thm:gamma-con} offers the $\varGamma$-convergence for the special choice
	\begin{equation}\label{eq:rho-nx}
		\begin{aligned}
			\rho_{\delta_n,x}(y,z)=\frac{1}{\delta_n^2}K_{\delta_n}(|x-y|)K_{\delta_n}(|x-z|).
		\end{aligned}
	\end{equation}
	To maintain this property in general situations, some restrictions on  $\rho_{\delta_n,x}(y,z)$ should be  
 imposed. Firstly, in order for $E_n(u,0)$ to capture information about the boundary, the kernel should rapidly decay or directly vanish when $|(y,z)-(x,x)|$ is large. Hence, we require $\rho_{\delta_n,x}$ to be compactly supported and the support set to be shrinking when $n$ tending to infinity, ensuring increasingly precise delineation for the boundary. Specifically, there should exist a sequence of positive constants $\{c_n\}$, $\lim_{n\rightarrow\infty}c_n=0$, such that $\rho_{\delta_n,x}$ is only nonzero when $|x-y|$, $|x-z|\leq c_n$. 
	
Moreover, as $n\to \infty$, we expect the convergence of the
minimizers 
 of $\{F_n\}$, which are confined only in  $L^2(\Omega)$, to a minimizer of $F$ in the $L^2$ norm. As the minimizer of $F$, we have $u\in H^1$ and $Tu\equiv 0$. Thus, intuitively speaking, the sequence of minimizers of $\{F_n\}$ should take small absolute value near the boundary. 
 To quantify this and to overcome the possible lack of regularity, 
 we pick a kernel $\hat{K}$ satisfying (K1)-(K3),  \gwy{$\hat{K}_\delta(s)=\frac{1}{\delta^d}\hat{K}(\frac{s^2}{\delta^2})$}, and
 a regularized form of  $\{u_n\}$:
		  \begin{equation}
			 \label{eq:tilde-u}
			 \tilde{u}_n (x)\coloneqq \frac{1}{\omega_{\delta_n}(x)}\int_\Omega \hat{K}_{\delta_n}(|x-y|)u_n(y)dy, \; \text{ and }\;\omega_{\delta_n}(x)\coloneqq \int_\Omega \hat{K}_{\delta_n}(|x-y|)dy.
		  \end{equation}
Note that the normalization coefficient 
$\omega_{\delta_n}(x)$ has uniformly positive lower and upper bounds \qd{with respect to $\delta_n$ as $\delta_n\to 0$}.  \qd{Should $\hat{K}_\delta(s)$  be $\frac{1}{\delta^d}
\hat{K}(\frac{s^2}{\delta^2}) $?}  Mollifications like \eqref{eq:tilde-u} have been  utilized in other works, for example, \cite{scott2023nonlocal}.
We then require a coercivity condition \[C_nE_n(u_n,0)\geq\lVert\tilde{u}_n\rVert_{L^2(\partial\Omega)}^2, \quad\forall u_n\in L^2(\Omega)\] for some 
$\hat{K}$  and positive constants $\{C_n\}$ satisfying $\lim_{n\rightarrow\infty}C_n=0$. 
Effectively, this means that the penalty  functional is placed on  matching with the boundary data, not directly by $u_n$, but a more regular $\tilde{u}_n$.
As a summary, we have the following theorem:
			 
		  \gwy{\begin{theorem}
			  \label{thm:gamma-con-general}
			  Suppose that $\Omega$ is a Lipschitz bounded domain in $\mathbb{R}^d$. $R$ is
     a kernel satisfying (K1)-(K3), \qd{and $\sigma_R$ is given by \eqref{eq:sigmaR}}.
     $\{\delta_n\}$, $\{c_n\}$, $\{C_n\}$ are sequences of positive constants tending to 0 as $n\rightarrow\infty$. $\rho_{\delta_n,x}$ is a kernel satisfying following two conditions:
			  \begin{enumerate}
				  \item[](compact support) $\rho_{\delta_n,x}$ is only nonzero when $|x-y|$, $|x-z|\leq c_n$.
				  \item[](coercivity) $C_nE_n(u_n,0)\geq \lVert\tilde{u}_n\rVert_{L^2(\partial\Omega)}^2$ for any $u_n\in L^2(\Omega)$,
			  \end{enumerate}
			  where $E_n$, $\{\tilde{u}_n\}$ are defined as \eqref{eq:boudary-term} and \eqref{eq:tilde-u} for some $\hat{K}$ satisfying (K1)-(K3).
			  Then, we have
			  \[F_n\stackrel{\varGamma}{\longrightarrow}\gwy{\sigma_RF}\quad\text{in }L^2(\Omega),\]
			  where $F_n,F$ are defined as \eqref{eq:dis-fun-general},\eqref{eq:con-fun-special}.
		  \end{theorem}}
		\begin{remark}
		    For general $p>1$, similar conclusion can be derived in the same way as in \cref{subsec:proof3}. However, we do not yet have a unified formulation, like \eqref{eq:boudary-term} for $p=2$, that also works for more general $p\neq 2$, i.e., a formulation that covers both  
            \begin{equation*}
            \begin{aligned}
      E_n^1(u,a)=\int_{\partial\Omega}\left|\frac{1}{\delta_n}\int_\Omega K_{\delta_n}(|x-y|)(u(y)-a(x))dy\right|^pdx,
                 \end{aligned}
            \end{equation*}
      and
                \begin{equation*}
            \begin{aligned}
                E_n^2(u,a)&=\frac{1}{\delta_n^p}\int_{\partial\Omega}\int_\Omega K_{\delta_n}(|x-y|)|u(y)-a(x)|^pdydx.
            \end{aligned}
            \end{equation*}
   Hence, for the purpose of conciseness, we only discuss the $p=2$ case in \cref{thm:gamma-con-general}.
		\end{remark}
  \gwy{\begin{remark}
  \label{remark1}
    The main contribution of \autoref{thm:gamma-con-general}, \qd{in comparison with \autoref{thm:gamma-con}, is to formulate} a more general boundary term. This general form contains a large class of kernel $\rho_{\delta_n,x}$ such as $\rho_{\delta_n,x}(y,z)=\frac{1}{\delta_n^2}K_{\delta_n}(|x-y|)\delta(|y-z
    |)$, \qd{where} $\delta(\cdot)$ means \qd{the Dirac-delta} measure. See more details in \autoref{subsec:proof3}. The subscript $x$ in \qd{the} kernel $\rho_{\delta_n,x}(y,z)$  means not only that it is centered at $x$ but also that other parts of kernel expression can depend on $x$. For example, \cite{tao2017nonlocal,wang2023nonlocal}
    actually considered a special case covered by Theorem 2.3,
 \[\rho_{\delta_n,x}(y,z)=\frac{2}{\delta^2\bar{\bar{\omega}}_{\delta_n}(x)}\bar{R}_{\delta_n}(|x-y|)\bar{R}_{\delta_n}(|x-z|)\]
 where $\bar{\bar{\omega}}_{\delta_n}$ is a bounded function depending on $x$. And in~\cite{shi2023nonlocal}, another special case is considered, 
 \begin{equation*}
     \rho_{\delta_n,x}(y,z)=\frac{4}{\delta_n^2\mu_{\delta_n}(x)}\bar{R}_{\delta_n}(|x-y|)\delta_n(|y-z|)
 \end{equation*}
where $\mu_{\delta_n}(x)$ is also a bounded function. 
  \end{remark}}
With the $\varGamma$-convergence,  the convergence of their minimizers folows.
		  \gwy{\begin{theorem}
			  \label{thm:gamma-con-minimizers}
			  Suppose that $\Omega$ is a Lipschitz bounded domain in $\mathbb{R}^d$. $1<p<\infty$ is a constant. $K,R$ are two kernel functions satisfying (K1)-(K3),      \qd{with $W$ satisfying  additionally \eqref{eq:normalW} and $\sigma_R$ being given by \eqref{eq:sigmaR}}. 
     $\{\delta_n\}$ is a sequence of positive constants tending to 0 as $n\rightarrow\infty$. $F_n,F$ are defined as \eqref{eq:dis-fun},\eqref{eq:con-fun}. Then any sequence $\{u_n\}\subset L^p(\Omega)$ satisfying
			  \[\lim_{n\rightarrow\infty}(F_n(u_n)-\inf_{u\in L^p(\Omega)}F_n(u))=0\]
			  is relatively compact in $L^p(\Omega)$ and
			  \[\lim_{n\rightarrow\infty}F_n(u_n)=\min_{u\in L^p(\Omega)}\gwy{\sigma_RF(u)}.\]
			  Furthermore, every cluster point of $\{u_n\}$ is a minimizer of $F$.
		  \end{theorem}
		
		  \begin{theorem}
				\label{thm:eigenvalue-minimizers}
				Suppose that $\Omega$ is a Lipschitz bounded domain in $\mathbb{R}^d$.  $K,R,W$ are three kernel functions satisfying (K1)-(K3) 
         \qd{with $W$ satisfying  additionally \eqref{eq:normalW} and $\sigma_R$ being given by \eqref{eq:sigmaR}}.
    $\{\delta_n\}$ is a sequence of positive constants tending to 0 as $n\rightarrow\infty$. $F_{e,n}^k,\tilde{F}_{e,n}^k$ and $F_e^k$ are defined as \eqref{eq:eigen-dis-fun}, \eqref{eq:normal-dis-fun} and \eqref{eq:eigen-cons-func}. Then any sequence $\{u_n\}\subset L^2(\Omega)$ satisfying
				\[\lim_{n\rightarrow\infty}(F_{e,n}^k(u_n)-\inf_{u\in L^2(\Omega)}F_{e,n}^k(u))=0\]
				is relatively compact in $L^2(\Omega)$ and	\[\lim_{n\rightarrow\infty}F_{e,n}^k(u_n)=\min_{u\in L^2(\Omega)}\gwy{\sigma_RF_{e}^k(u)}.\]
				Furthermore, every cluster point of $\{u_n\}$ is a minimizer of $F_{e}^k$. The conclusions still hold if $\{F_{e,n}^k\}$ is replaced by $\{\tilde{F}_{e,n}^k\}$.
		  \end{theorem}
		
		  \begin{theorem}
				\label{thm:gamma-con-general-minimizers}
				Suppose that $\Omega$ is a Lipschitz bounded domain in $\mathbb{R}^d$. $R$ is a kernel satisfying (K1)-(K3),
     \qd{and $\sigma_R$ is given by \eqref{eq:sigmaR}}.
    $\{\delta_n\}$, 
    $\{c_n\}$, $\{C_n\}$ are sequences of positive constants tending to 0 as $n\rightarrow\infty$. $\rho_{\delta_n,x}$ is a kernel satisfying following two conditions:
				\begin{enumerate}
					\item[](compact support) $\rho_{\delta_n,x}$ is only nonzero when $|x-y|$, $|x-z|\leq c_n$.
					\item[](coercivity) $C_nE_n(u_n,0)\geq \lVert\tilde{u}_n\rVert_{L^2(\partial\Omega)}^2$ 
   for any $u_n\in L^2(\Omega)$,
				\end{enumerate}
				where $E_n$ and $\{\tilde{u}_n\}$ are defined as \eqref{eq:boudary-term} and  \eqref{eq:tilde-u} for some $\hat{K}$ satisfying (K1)-(K3), $F_n$ and $F$ are defined as \eqref{eq:dis-fun-general},\eqref{eq:con-fun-special}.
				Then any sequence $\{u_n\}\subset L^2(\Omega)$ such that
				\[\lim_{n\rightarrow\infty}(F_n(u_n)-\inf_{u\in L^2(\Omega)}F_n(u))=0\]
				is relatively compact in $L^2(\Omega)$ and
				\[\lim_{n\rightarrow\infty}F_n(u_n)=\min_{u\in L^2(\Omega)}\gwy{\sigma_RF(u)}.\]
				Futhermore, every cluster point of $\{u_n\}$ is a minimizer of $F$.
		  \end{theorem}}

\gwy{
\section{Quantitative discussion}
\label{sec:quant-discussion}
    All results in this paper are based on $\Gamma$-\\
    convergence. As shown by \cref{lem:minimizer}, this concept can lead to the convergence of minimizers, but do not provide information about the convergence rate. Such weak conclusions allow us to consider a general class of nonlocal model. It is absolutely better if one can derive some quantitative results without strengthen our assumptions. But it may be difficult since the kernels of boundary and interior term do not need any connection. Nevertheless, in some special cases, if the kernels are carefully designed and have good linkage, one can obtain a specific nonlocal model with stronger quantitative properties. For example, as mentioned in \cref{remark1}, \cite{wang2023nonlocal} studies the nonlocal model \eqref{eq:dis-fun-general} with 
    \[\rho_{\delta_n,x}(y,z)=\frac{2}{\delta^2\bar{\bar{\omega}}_{\delta_n}(x)}\bar{R}_{\delta_n}(|x-y|)\bar{R}_{\delta_n}(|x-z|)\]
    where 
    \[\bar{\bar{\omega}}_{\delta_n}(x)\coloneqq \int_\Omega\bar{\bar{R}}_{\delta_n}(|x-y|)dy.\]
     And demonstrate the first order $H^1$ convergence of the minimizers (Theorem 3.2 in \cite{wang2023nonlocal}). Select
     \begin{equation*}
     \rho_{\delta_n,x}(y,z)=\frac{4}{\mu_{\delta_n}(x)}\bar{R}_{\delta_n}(|x-y|)\delta_n(|y-z|)
 \end{equation*}
where $\mu_{\delta_n}(x)\coloneqq\min\{2\delta_n,\max\{\delta_n^2,d(x)\}\}$ and $d(x)=\min_{y\in\partial\Omega}|x-y|$. It is shown in \cite{shi2023nonlocal} that \eqref{eq:dis-fun-general} with such $\rho_{\delta_n,x}$ fulfills maximun principle and second order convergence for the minimizers (Theorem 5.1 in \cite{shi2023nonlocal}).

Moreover, except for the convergence rate of minimizers, one may also directly explore the convergence rate of nonlocal functional. In~\cite{chambolle2020convergence}, the author prove that $\frac{\mathcal{F}_h-\mathcal{F}_0}{h^2}$ $\Gamma$-converges to a nonzero functional. $\mathcal{F}_0$, $\mathcal{F}_h$ are defined as follows,
    \begin{equation*}
        \begin{aligned}
            &\mathcal{F}_0(u)\coloneqq\int_{\mathbb{R}^d}\int_{\mathbb{R}^d}K(z)f\left(\left|\nabla u(x)\cdot\frac{z}{|z|}\right|\right)dzdx\\
            &\mathcal{F}_h(u)\coloneqq\int_{\mathbb{R}^d}\int_{\mathbb{R}^d}K_h(y-x)f\left(\frac{|u(y)-u(x)|}{|y-x|}\right)dydx
        \end{aligned}
    \end{equation*}
    where $f(\cdot)$ is a convex function fulfills some conditions. This conclusion can be understood as nonlocal energy $\mathcal{F}_h$ second order converges to $\mathcal{F}_0$. If we choose $f(x)=x^p$ with $p>1$, $\mathcal{F}_h$ is quite similar with the interior term in our nonlocal model. The differences are that we change $\frac{1}{|y-x|}$ into $\frac{1}{h}$ and $\mathbb{R}^d$ into a bounded domain $\Omega$. We believe that these differences are not essential and that similar conclusion can be demonstrated in the cases we consider. However, much of the proof of this conclusion may be devoted to the treatment of interior terms, and the specific form of kernel in the boundary term has little influence. This goes against our topic about boundary terms and may be more suitable as future work. Specifically, consider the nonlocal model (2.10) for Theorem 2.3,
    \begin{equation*}
		\begin{aligned}
			F_n(u)&=\frac{1}{\delta_n^2}\int_\Omega\int_\Omega R_{\delta_n}(|x-y|)|u(x)-u(y)|^2dxdy
			+E_n(u,0)\\
   &=:I_{\delta_n}(u)+E_n(u,0).
		\end{aligned}
    \end{equation*}
    Define
    \[I_0(u)=\int_{\mathbb{R}^d}\int_\Omega R(|z|)|\nabla u(x)\cdot z|^pdxdz.\]
    Suppose that with a similar method in~\cite{chambolle2020convergence}, one can prove that $\frac{I_{\delta_n}-I_0}{\delta_n^2}$ $\Gamma$-converges to a nonzero functional $\mathcal{I}(u)$. With the coercivity condition in Theorem 2.3, there exists a sequence of positive constants $\{C_n\}$ tending to zero such that
    \[C_nE_n(u,0)\geq\lVert\tilde{u}\rVert_{L^2(\partial\Omega)}\]
    for all $u\in L^2(\Omega).$ Combine with Lemma 4.3, it can be proved that $\frac{F_n-I_0}{\delta_n^2}=\frac{I_{\delta_n}-I_0}{\delta_n^2}+\frac{1}{\delta_n^2}E_n(u,0)$ $\Gamma$-converges to
    \begin{equation*}
		\tilde{\mathcal{I}}(u)=\left\{\begin{array}{cc}
			\mathcal{I}(u)&\text{if }Tu\equiv 0\text{ on }\partial\Omega,\\
			\infty&\text{otherwise}
		\end{array}\right.
    \end{equation*}
    when $\frac{C_n}{\delta_n^2}\rightarrow 0.$ So the specific form of boundary term $E_n$ does not matter. While proving the conclusion that $\frac{I_{\delta_n}-I_0}{\delta_n^2}$ $\Gamma$-converges to a nonzero functional 
    requires much more effort, especially for analyze near the boundary of $\Omega$.

    Even if we prove the $\Gamma$-convergence above, there is still a gap for the convergence of minimizers. However, such convergence of functionals may derive the convergence of minimizers with some stronger assumptions. For example, a simple situation is that $\frac{I_{\delta_n}-I_0}{\delta_n^2}$ uniformly converges to a bounded functional $\mathcal{I}(u)$ and $I_0$ is strongly convex. In this case, we can obtain the second order convergence of the minimizers. This gap is also one of the obstacles we hope to overcome in future work.
} 
	\section{Preliminaries}
	\label{sec:preliminaries}
For easy reference, we first recall
 the concept of $\varGamma$-\\
 convergence 
or functionals. We then present a technical lemma concerning the nonlocal energy with the kernel rescaled by different constants.
 \subsection{\texorpdfstring{$\varGamma$}--Convergence}
The $\varGamma$-convergence 
proposed by De Giorgi is often used to study 
 the convergence of minimizers of functionals under compactness assumptions. 
 More detailed overviews about the $\varGamma$-convergence can be found in~\cite{braides2002gamma,dal2012introduction}.
 
	\begin{definition}[$\varGamma$-convergence]
		\label{def:gamma}
		Let $X$ be a metric space and $F_n:X\rightarrow\mathbb{R}\cup\{-\infty,\infty\}$ be a sequence of functionals on $X$. We say that $F_n$ $\varGamma$-converges to $F:X\rightarrow\mathbb{R}\cup\{-\infty,\infty\}$,
  	which is also denoted by $F_n\stackrel{\Gamma}{\rightarrow}F(n\rightarrow\infty)$, if
		\begin{enumerate}
			\item[(1)](liminf inequality) for any sequence $\{x_n\}_{n\in\mathbb{N}}\subset X$ converging to $x\in X$, we have
		$\displaystyle \liminf_{n\rightarrow\infty}F_{n}(x_n)\geq F(x).$
			\item[(2)](limsup inequality) for any $x\in X$, there exists a sequence $\{x_n\}_{n\in\mathbb{N}}\subset X$ converging to $x$ such that
$\displaystyle \limsup_{n\rightarrow\infty}F_n(x_n)\leq F(x).$
		\end{enumerate}
	\end{definition}
   The following lemma reveals the connection between $\varGamma$-convergence and the convergence of minimizers, which is also applied in \cite{garcia2016continuum,slepcev2019analysis,roith2022continuum,gan2022non}.
We include here for completeness and easy reference.
	\begin{lemma}[Convergence of minimizers]
		\label{lem:minimizer}
		Let $X$ be a metric space and $F_n:X\rightarrow[0,\infty]$ $\varGamma$-converges to $F:X\rightarrow[0,\infty]$ which is not identically $\infty$. If there exists a relatively compact sequence $\{x_n\}_{n\in\mathbb{N}}\subset X$ such that
		\[\lim_{n\rightarrow\infty}(F_n(x_n)-\inf_{x\in X}F_n(x))=0\]
		then we have
		\[\lim_{n\rightarrow\infty}\inf_{x\in X}F_n(x)=\min_{x\in X}F(x)\]
		and any cluster point of $\{x_n\}_{n\in\mathbb{N}}$ is a minimizer of $F$.
	\end{lemma}
	\begin{proof}
		For any $y\in X$, we know that there exists a sequence $\{y_n\}_{n\in\mathbb{N}}\subset X$ satisfying the limsup inequality. So we have
		\[F(y)\geq\limsup_{n\rightarrow\infty}F_n(y_n)\geq\limsup_{n\rightarrow\infty}\inf_{x\in X}F_n(x)\]
		which yields
		\[\min_{x\in X}F(x)\geq\limsup_{n\rightarrow\infty}\inf_{x\in X}F_n(x)\]
		On the other hand, consider the sequence $\{x_n\}_{n\in\mathbb{N}}\subset X$ mentioned in the assumption, let $\tilde{x}$ be one of the cluster points of $\{x_n\}_{n\in\mathbb{N}}$, using the liminf ineuqality, we get
		\[\liminf_{n\rightarrow\infty}\inf_{x\in X}F_n(x)=\liminf_{n\rightarrow\infty}F_n(x_n)\geq F(\tilde{x})\geq\min_{x\in X}F(x)\]
		Therefore,
		\[\limsup_{n\rightarrow\infty}\inf_{x\in X}F_n(x)\leq\min_{x\in X}F(x)\leq F(\tilde{x})\leq\liminf_{n\rightarrow\infty}\inf_{x\in X}F_n(x)\]
		and we can get the conclusion.
	\end{proof}
	\subsection{Relation between kernels of different scales}
The following technical\\ lemma clarifies that when the kernel is rescaled by a constant factor, the 
 nonlocal energy remains
 uniformly controlled by the original one. This conclusion is also indispensable in \cite{shi2017convergence,gan2022non,du2022nonlocal}.
	\begin{lemma}
		\label{lem:kernel_diff_scales}
		\gwy{Let $R$ be a kernel satisfying (K1)-(K3). $p>1$, $m>0$ are finite constants and $\Omega$ is a Lipschitz bounded domain in $\mathbb{R}^d$. Then there exists a constant $C$ depending on $m$, such that
		for all $\delta>0$ and $u\in L^p(\Omega)$,}
		\begin{equation*}
			\begin{aligned}
		&		\int_\Omega\int_\Omega R_{\delta}(|x-y|)|u(x)-u(y)|^pdxdy\\
    &\qquad  \leq C \int_\Omega\int_\Omega R_{m\delta}(|x-y|)|u(x)-u(y)|^pdxdy.
			\end{aligned}
		\end{equation*}
  \end{lemma}
		\begin{proof}
			When $m\geq 1$, 
  by the monotone descreasing property of $R$, 
  we have 
  \begin{equation*}
	R\left(\frac{|x-y|^2}{(m^{-1}\delta)^2}\right)\leq   R\left(\frac{|x-y|^2}{\delta^2}\right)	\leq R\left(\frac{|x-y|^2}{(m\delta)^2}\right),
			\end{equation*}
Thus, 
   			\begin{equation*}
   m^{-d}  R_{m^{-1}\delta}(|x-y|)	\leq    R_{\delta}(|x-y|)	\leq 
     m^d R_{m\delta}(|x-y|).
			\end{equation*}
The conclusion for $m\geq 1$ then follows easily and the case for $m<1$ can be obtained from a telescoping argument, similar to those presented in \cite{shi2017convergence,du2022nonlocal,scott2023nonlocal}.
 \end{proof}

    
	\section{\texorpdfstring{$\varGamma$}--Convergence of nonlocal functionals}
	\label{sec:proof}
		\subsection{Nonlocal model for Dirichlet problem}
		\label{subsec:proof1}
To prepare for
  the proof of \cref{thm:gamma-con}, we first present
  some lemmas.
   The first one is about the property of the convolution between a kernel and a sequence of $L^p$ functions which has a limit. It is well known that for a sequence of positive constants $\delta_n\rightarrow 0$, a $L^1$ kernel function $R$ and a $L^p$ function $u$, the equality
		\begin{equation}
  \label{eq:conv-limit}
		\lim_{n\rightarrow\infty}\lVert R_{\delta_n}*u-C_R u\rVert_{L^p}=0,
		\end{equation}
		holds for any $1\leq p<\infty$, where $R_{\delta_n}$ is the scaled kernel of $R$ and 
  $C_R\coloneqq \int_{\mathbb{R}^d}R(|y|^2)dy$ 
  is a constant only 
dependent on $R$.  \qd{Replacing $u$ in the above by 
  a sequence $\{u_n\}$ that converges to $u$}, we have a similar conclusion: 
		\begin{lemma}
			\label{lem:fou}
			Suppose that $\Omega$ is a domain in $\mathbb{R}^d$. $R$ is a kernel function satisfying $(K1)\text{-}(K3)$. $1< p<\infty$. For a positive constant $\delta$, $R_\delta(x)\coloneqq\frac{1}{\delta^d}R(\frac{x^2}{\delta^2})$. $\delta_n\rightarrow 0$ and $u_n\rightarrow u$ in $L^p(\Omega)$ as $n\rightarrow\infty$. Then,
			\[\left\lVert\int_\Omega R_{\delta_n}(|x-y|)(u_n(x)-u_n(y))dy\right\rVert_{L^p(\Omega)}\longrightarrow 0,\text{ as }n\rightarrow \infty.\]
		\end{lemma}
		\begin{proof}
Let  $C_{R,n}(x)\coloneqq 
  \int_{\Omega} R_{\delta_n}(|x-y|)dy$ for $x\in \Omega$. Note that
  \gwy{\begin{equation*}
      \begin{aligned}
& \|R_{\delta_n}*u_n-C_{R,n} u_n\|_{L^p(\Omega)}
 \leq \|R_{\delta_n}* (u_n-u)\|_{L^p(\Omega)}  +\|C_{R,n} (u- u_n)\|_{L^p(\Omega)}\\
&\qquad +\| (C_{R,n}-C_R )u \|_{L^p(\Omega)}
+\|R_{\delta_n}* u - C_R u\|_{L^p(\Omega)}.    
\end{aligned}
  \end{equation*}}

The first two terms go to $0$ by the convergence of $u_n$ to $u$ and boundedness of convolution with $R_{\delta_n}$ and the uniform bound of the function $\|C_{R,n}\|_{L^\infty(\Omega)}\leq C_R$, \gwy{while the third term goes to zero since $C_{R,n}(x)=C_R$ except for $x$ in the layer $\{x\in \Omega, \text{dist}(x,\Omega^c)< r_R \delta_n\}$ whose measure goes to $0$.} The last term follows from \eqref{eq:conv-limit}.
  \end{proof}

  		In the functionals \eqref{eq:dis-fun} under consideration, the $\varGamma$-convergence of the first term has been studied (\cite{slepcev2019analysis} Lemma 4.6). We will directly use the liminf part, which is listed as follows:
		\begin{lemma}
			\label{lem:gamma-con}
			Let $\Omega,R,\sigma_R$ satisfy the same conditions as \cref{thm:gamma-con}. $1<p<\infty$. $u_n\rightarrow u$ in $L^p(\Omega)$,  
   $\delta_n\rightarrow 0$, then 
			\[\liminf_{n\rightarrow\infty}\frac{1}{\delta_n^p}\int_\Omega\int_\Omega R_{\delta_n}(|x-y|)|u_n(x)-u_n(y)|^pdxdy\geq \gwy{\sigma_R}E(u),\]
			where
			\gwy{\begin{equation*}
				E(u)=\left\{\begin{array}{cc}
					\int_\Omega|\nabla u(x)|^pdx&\text{if }u\in W^{1,p},\\
					\infty&\text{otherwise.}
				\end{array}\right.
			\end{equation*}}
		\end{lemma}
		The main difference between \cref{thm:gamma-con} and the Lemma 4.6 in \cite{slepcev2019analysis} is the additional term about the boundary $\partial\Omega$. To resolve it, the trace theorem that the $L^p(\partial\Omega)$ norm of the trace of a $W^{1,p}(\Omega)$ function can be controlled by the $W^{1,p}$ norm is useful. While the strong $W^{1,p}$ convergence from $u_n$ to $u$ does not follow directly from the derivation of the liminf inequality, it turns out that, with the following lemma, the weak $W^{1,p}$ convergence is enough.
		\begin{lemma}
			\label{lem:trace} 
			Suppose that $\Omega$ is a Lipschitz bounded domain and $1<p<\infty$. Let $\{u_n\}\subset W^{1,p}(\Omega)$ 
   satisfy that $\sup_n\lVert u_n\rVert_{W^{1,p}(\Omega)}<\infty,$ and $u_n\rightarrow u$ in $L^p(\Omega)$ for some $u\in W^{1,p}(\Omega)$ with $\lVert Tu_n\rVert_{L^p(\partial\Omega)}\rightarrow 0$ as $n\rightarrow\infty$. Then we have 
   $Tu=0$ on $\partial\Omega$, in the sense of trace space.
		\end{lemma}
		\begin{proof}
 By the reflexivity of  $W^{1,p}(\Omega)$, the trace theorem(see, for example , in \cite{ern2004theory}) and compact imbeddings of Sobolev spaces, we can see that $u$ is the weak limit
of $u_n$ in $W^{1,p}(\Omega)$, and
$Tu$ is both the weak limit of $Tu_n$ in $W^{1-1/p,p}(\partial\Omega)$ and the strong limit in $L^p(\partial\Omega)$. Thus $Tu=0$.
\end{proof}

		When processing the boundary term, as stated in \cref{sec: assump and result}, we actually transform $u_n$ into a more regular form $\{\tilde{u}_n\}$ defined as \eqref{eq:tilde-u}. For the gradient of $\tilde{u}_n$, we have the following $L^p$-estimate, see similar results presented in \cite{scott2023nonlocal}: 
		\begin{lemma}
			\label{lem:L2-estimation}
			\gwy{Let $\Omega,R$ satisfy the same conditions as \cref{thm:gamma-con}. $1<p<\infty$. $\tilde{u}_n$ is defined as \eqref{eq:tilde-u}. We have the following estimate about $\nabla\tilde{u}_n$,}
			\[\lVert\nabla \tilde{u}_n\rVert_{L^p(\Omega)}\leq \frac{C}{\delta_n}\left(\int_\Omega\int_\Omega R_{\delta_n}(|x-y|)|u_n(x)-u_n(y)|^pdxdy\right)^\frac{1}{p},\]
			\gwy{where $C$ is a constant depending only on $\hat{K},R$ and $\Omega$}.  
		\end{lemma}
		\begin{proof}
			By the definitions,
			\begin{equation*}
				\begin{aligned}				 &\nabla\tilde{u}_n(x) =\frac{1}{\omega_{\delta_n}(x)^2}\bigg(\int_\Omega\int_\Omega \hat{K}_{\delta_n}(|x-y|)\nabla_x \hat{K}_{\delta_n}(|x-z|)u_n(z)\\
					&\qquad\quad - \nabla_x\hat{K}_{\delta_n}(|x-y|) \hat{K}_{\delta_n}(|x-z|)u_n(z)dydz\bigg)\\
				&\quad	=\frac{1}{\omega_{\delta_n}(x)^2}\left(\int_\Omega\int_\Omega  \nabla_x\hat{K}_{\delta_n}(|x-y|) \hat{K}_{\delta_n}(|x-z|)(u_n(y)-u_n(z))dydz\right).
				\end{aligned}
			\end{equation*}
			Therefore, using the H\"older inequality,
			\gwy{\begin{equation*}
				\begin{aligned}			&\lVert\nabla\tilde{u}_n\rVert_{L^p}\leq C_1\left\lVert\int_\Omega\int_\Omega  \nabla_x\hat{K}_{\delta_n}(|x-y|) \hat{K}_{\delta_n}(|x-z|)(u_n(y)-u_n(z))dydz\right\rVert_{L^p}\\
     &\; \leq\frac{C_2}{\delta_n^\frac{1}{p^*}}\left\{\int_\Omega\int_\Omega\left(\int_\Omega|\nabla_x\hat{K}_{\delta_n}(|x-y|)| \hat{K}_{\delta_n}(|x-z|)dx\right)|u_n(y)-u_n(z)|^pdydz\right\}^\frac{1}{p}.
				\end{aligned}
			\end{equation*}
			where $\frac{1}{p^*}=1-\frac{1}{p}$}. Denote \[
   \tilde{K}_{\delta_n}(y,z) \coloneqq
   \int_\Omega\nabla_x\hat{K}_{\delta_n}(|x-y|) \hat{K}_{\delta_n}(|x-z|)dx.\] Recalling the condition $\mathrm{supp}(\hat{K})\subset[0,r_{\hat{K}}]$, it is obvious that $\tilde{K}_{\delta_n}(y,z)=0$ when $|y-z|>2\delta_nr_{\hat{K}}$. According to the regularity of kernel $\hat{K}$ and $R$, we may assume that there exists some constants $k_1,k_2,r_1>0$ such that
			\[\hat{K}(x)\leq k_1,\ |\hat{K}'(x)|\leq k_2,\ \forall x\geq 0, \quad\text{and}\quad
			R(x)\geq r_1,\ \forall x\in[0,\frac{r_R}{2}].\]
			For any $y,z$ with $|y-z|\leq 2\delta_nr_{\hat{K}}$,
			\begin{equation*}
				\begin{aligned}
					\tilde{K}_{\delta_n}(y,z)&\leq \frac{k_2}{\delta_n^{d+1}}\int_\Omega \hat{K}_{\delta_n}(|x-z|)dx
					\leq \frac{k_2}{\delta_n^{d+1}}\int_{\mathbb{R}^d} \hat{K}_{\delta_n}(|x-z|)dx\\
					&\leq \frac{C_{\hat{K}}}{\delta_n^{d+1}}
				\leq \frac{C_{\hat{K},R}}{\delta_n}R_{\frac{4r_{\hat{K}}}{r_R}\delta_n}(|y-z|).
				\end{aligned}
			\end{equation*}
			Hence,
			\begin{equation*}
				\begin{aligned}
					\lVert\nabla\tilde{u}_n\rVert_{L^p}&\leq\frac{C_2}{\delta_n^\frac{1}{p^*}}\left\{\int_\Omega\int_\Omega \tilde{K}_{\delta_n}(y,z)|u_n(y)-u_n(z)|^pdydz\right\}^\frac{1}{p}\\
					&\leq\frac{C_3}{\delta_n}\left\{\int_\Omega\int_\Omega R_{\frac{4r_{\hat{K}}}{r_R}\delta_n}(|y-z|)|u_n(y)-u_n(z)|^pdydz\right\}^\frac{1}{p}\\
					&\leq \frac{C}{\delta_n}\left\{\int_\Omega\int_\Omega R_{\delta_n}(|y-z|)|u_n(y)-u_n(z)|^pdydz\right\}^\frac{1}{p}.\\
				\end{aligned}
			\end{equation*}
			The last inequality holds owing to \cref{lem:kernel_diff_scales}.
		\end{proof}
		With the preparation above, we can start to prove \cref{thm:gamma-con}. Firstly, we simplify the problem into the situation of $a\equiv 0$. To do this, we consider a $W^{1,p}$ function $v$ whose trace is $a$. 
  Then we can transform 
  the functionals by translation,
		\begin{equation*}
			F^v(u)\coloneqq F(v+u),\ F_n^v\coloneqq F_n(v+u).
		\end{equation*}
		Specifically,
		\begin{equation}
			\label{eq:dis-fun2}
			\begin{aligned}
				F_n^v(u)=&\frac{1}{\delta_n^p}\int_\Omega\int_\Omega R_{\delta_n}(|x-y|)|u(x)+v(x)-u(y)-v(y)|^pdxdy\\
				&+\int_{\partial\Omega}\left|\frac{1}{\delta_n}\int_\Omega K_{\delta_n}(|x-y|)(a(x)-u(y)-v(y))dy\right|^pdx
			\end{aligned}
		\end{equation}
  and
		\begin{equation}
			\label{eq:con-fun2}
			F^v(u)=\left\{\begin{array}{cc}
				\int_\Omega|\nabla u(x)+\nabla v(x)|^pdx&\text{if }u\in W^{1,p},Tu\equiv 0\text{ on }\partial\Omega,\\
				\infty&\text{otherwise.}
			\end{array}\right.
		\end{equation}
		It is obvious that the $\varGamma$-convergence from $F_n$ to $F$ is equivalent to the one from $F_n^v$ to $F^v$. And we have the following lemma to show that the latter
		is similar to \cref{thm:gamma-con} when $a\equiv 0$. 
		\begin{lemma}
			\label{lem:approx}
			Let $F^v$, $F_n^v$ be defined as \eqref{eq:dis-fun2},\eqref{eq:con-fun2}. Then it is sufficient for \cref{thm:gamma-con} to prove the $\varGamma$-convergence from $\tilde{F}_n^v$ to $\sigma_RF^v,$ where
			\begin{equation}
				\label{eq:dis-fun3}
				\begin{aligned}
					\tilde{F}_n^v(u)=&\frac{1}{\delta_n^p}\int_\Omega\int_\Omega R_{\delta_n}(|x-y|)|u(x)+v(x)-u(y)-v(y)|^pdxdy\\
					&+\int_{\partial\Omega}\left|\frac{1}{\delta_n}\int_\Omega K_{\delta_n}(|x-y|)u(y)dy\right|^pdx.
				\end{aligned}
			\end{equation}
		\end{lemma}
		\begin{proof}
			According to the definitions, we only need to demonstrate that
			\[\left\lVert\frac{1}{\delta_n}\int_\Omega K_{\delta_n}(|x-y|)(v(y)-a(x))dy\right\rVert_{L^p(\partial\Omega)}\rightarrow 0\;\text{ as }\; n\rightarrow \infty.\] 
            Consider the extension $\bar{v}$ of $v$ on the convex hull $\bar{\Omega}$ of $\Omega$ and a sequence of $C^1$ approximation $\{v_n\}$ of $\bar{v}$, such that
            $\lVert v_n-\bar{v}\rVert_{W^{1,p}(\bar{\Omega})}=o(\delta_n^{1+\frac{d}{p}})$.
            By the trace theorem, this gives 
            $\lVert v_n-a\rVert_{L^p(\partial\Omega)}=o(\delta_n^{1+\frac{d}{p}})=o(\delta_n)$.
            Note that
            \begin{equation*}
               \begin{aligned}
                &\left\lVert\frac{1}{\delta_n}\int_\Omega K_{\delta_n}(|x-y|)(v(y)-a(x))dy\right\rVert_{L^p(\partial\Omega)}\\
                &\quad \leq\left\lVert\frac{1}{\delta_n}\int_\Omega K_{\delta_n}(|x-y|)(v_n(x)-a(x))dy\right\rVert_{L^p(\partial\Omega)}\\
                &\quad +\left\lVert\frac{1}{\delta_n}\int_\Omega K_{\delta_n}(|x-y|)(v_n(y)-v(y))dy\right\rVert_{L^p(\partial\Omega)}\\
                &\quad +\left\lVert\frac{1}{\delta_n}\int_\Omega K_{\delta_n}(|x-y|)(v_n(y)-v_n(x))dy\right\rVert_{L^p(\partial\Omega)}\\
                &\quad=:I_1+I_2+I_3.
            \end{aligned} 
            \end{equation*}
            For the first and the second term, they tend to zero since $\{v_n\}$ converges to $v$ at a sufficiently rapid rate.
            \begin{equation*}
             \begin{aligned}
                I_1&=\left\lVert\frac{1}{\delta_n}\int_\Omega K_{\delta_n}(|x-y|)dy(v_n(x)-a(x))\right\rVert_{L^p(\partial\Omega)}\\
                &=O(\frac{1}{\delta_n}\lVert v_n(x)-a(x)\rVert_{L^p(\partial\Omega)})\\
                &=o(1).
            \end{aligned}   
            \end{equation*} and
            \begin{equation*}
            \begin{aligned}
                I_2^p&=\int_{\partial\Omega}\left|\frac{1}{\delta_n}\int_\Omega K_{\delta_n}(|x-y|)(v_n(y)-v(y))dy\right|^pdx\\
                &=O\left(\int_{\partial\Omega}\int_\Omega \frac{1}{\delta_n^p}K_{\delta_n}(|x-y|)|v_n(y)-v(y)|^pdydx\right)\\
                &=O\left(\frac{1}{\delta_n^{p+d}}\lVert v_n-v\rVert_{L^p(\Omega)}^p\right)\\
                &=o(1).
            \end{aligned}
            \end{equation*}  
            For the last term, using the Taylor expansion,
    we have
        \begin{equation*}
            \begin{aligned}
                I_3^p
                &=O\left(\int_{\partial\Omega}\int_\Omega K_{\delta_n}(|x-y|)\frac{1}{\delta_n^p}|v_n(x)-v_n(y)|^pdydx\right)\\
                &= O\left(\int_{\partial\Omega}\int_{\bar{\Omega}} K_{\delta_n}(|x-y|)\frac{1}{\delta_n^p}|v_n(x)-v_n(y)|^pdydx\right)\\
&=O\left(\int_{\partial\Omega}\int_{\bar{\Omega}}\int_0^1 K_{\delta_n}(|x-y|)\frac{1}{\delta_n^p}|(y-x)\cdot\nabla v_n(x+t(y-x))|^pdtdydx\right)\\
&=O\left(\int_{\partial\Omega}\int_{|z|\leq r_K, x+\delta_n z\in{\bar{\Omega}}}\int_0^1 K(|z|^2)|z\cdot\nabla v_n(x+t\delta_n z)|^pdtdzdx\right)\\
&=O\left(\int_{\partial\Omega}\int_0^1\int_{|z|\leq r_K, x+\delta_n z\in\bar{\Omega}} |\nabla v_n(x+t\delta_n z)|^pdzdtdx\right)\\
        &=O(\lVert\nabla v_n\rVert_{L^p(\tilde{\Omega}_{\delta_n})}),
            \end{aligned}
        \end{equation*}
            where $\tilde{\Omega}_{\delta_n}\coloneqq\{x\in{\bar{\Omega}}\big|dist(x,\partial\Omega)\leq\delta_n r_K\}$ and the last equation holds since 
            \[\int_{|z|\leq r_K, x+\delta_n z\in{\bar{\Omega}}} |\nabla v_n(x+t\delta_n z)|^pdz\leq \lVert\nabla v_n\rVert_{L^p(\tilde{\Omega}_{\delta_n})}\]
            for all fixed $x\in\partial\Omega,\ t\in(0,1]$.
            Note that the measure of $\tilde{\Omega}_{\delta_n}$ tends to zero.  Hence,
            \[\lVert\nabla v_n\rVert_{L^p(\tilde{\Omega}_{\delta_n})}\leq \lVert\nabla \bar{v}\rVert_{L^p(\tilde{\Omega}_{\delta_n})}+\lVert \bar{v}-v_n\rVert_{W^{1,p}({\bar{\Omega}})}=o(1)\]
		\end{proof}
To complete the proof of \cref{thm:gamma-con}, we show the $\varGamma$-convergence 
of  $\tilde{F}_n^v$ to $F^v$. The proof 
is divided into two parts according to \cref{def:gamma}. 
		\begin{lemma}[the liminf inequality]
			\label{lem:liminf}
			Suppose that $\Omega$ is a Lipschitz bounded domain in $\mathbb{R}^d$. $K,R$ are two kernel functions satisfying (K1)-(K3). $1<p<\infty$. $\delta_n\rightarrow 0$ as $n\rightarrow\infty$. Then for any $u_n\rightarrow u$ in $L^p(\Omega)$,
			\[\liminf_{n\rightarrow\infty}\tilde{F}_n^v(u_n)\geq \sigma_RF^v(u),\]
			where $\tilde{F}_n^v,F^v$ are defined as \eqref{eq:dis-fun3},\eqref{eq:con-fun2}.
		\end{lemma}
		\begin{proof}
			\gwy{Without loss of generality, we can suppose that $\liminf_{n\rightarrow\infty}\tilde{F}_n^v(u_n)<\infty$.} From \cref{lem:gamma-con}, we have  
			\[\liminf_{n\rightarrow\infty}\frac{1}{\delta_n^p}\int_\Omega\int_\Omega R_{\delta_n}(|x-y|)|u_n(x)+v(x)-u_n(y)-v(y)|^pdxdy\geq \sigma_RE(u+v).\]
			which also yields that $u\in W^{1,p}$.
			We then claim that $Tu\equiv 0$ on $\partial\Omega$. 
   Once the claim is verified, we get
		\begin{equation*}
		\begin{aligned}
					&\liminf_{n\rightarrow\infty}\tilde{F}_n^v(u_n)\\
					\geq&\liminf_{n\rightarrow\infty}\frac{1}{\delta_n^p}\int_\Omega\int_\Omega R_{\delta_n}(|x-y|)|u_n(x)+v(x)-u_n(y)-v(y)|^pdxdy\\
					\geq &\sigma_RE(u+v)\\
					= &\sigma_RF^v(u)
				\end{aligned}
			\end{equation*}
			\gwy{and the proof is complete.} To 
  show the claim, note that we also have 
 \[  \liminf_{n\rightarrow\infty}\int_{\partial\Omega}\left|\frac{1}{\delta_n}\int_\Omega K_{\delta_n}(|x-y|)u_n(y)dy\right|^pdx<\infty.\]
			\gwy{Define
   \begin{equation*}
			 \tilde{u}_n (x)\coloneqq \frac{1}{\omega_{\delta_n}(x)}\int_\Omega K_{\delta_n}(|x-y|)u_n(y)dy, \; \text{ and }\;\omega_{\delta_n}(x)\coloneqq \int_\Omega K_{\delta_n}(|x-y|)dy.
    \end{equation*}
The above property implies
   the fact that there exists a subsequence (still denoted by $\{\tilde{u}_n\}$ for simplification of notations) of $\{\tilde{u}_n\}$ satisfying}
			\[\lim_{n\rightarrow\infty} \lVert T\tilde{u}_n\rVert_{L^p(\partial\Omega)}=0\]
			owing to the positive lower bound of $\omega_{\delta_n}(x)$. For the $W^{1,p}$ estimation, firstly,
			\begin{equation*}
				\begin{aligned}
					\lVert\tilde{u}_n\rVert_{L^p(\Omega)}&\leq C_1\left(\int_\Omega\int_\Omega K_{\delta_n}(|x-y|)u_n^p(y)dydx\right)^\frac{1}{p}
					\leq C_2\lVert u_n\rVert_{L^p(\Omega)},
				\end{aligned}
			\end{equation*}
			which yields that the $L^p(\Omega)$ norm of $\{\tilde{u}_n\}$ is uniformly bounded. Moreover, 
			\begin{equation*}
				\begin{aligned}
					&\limsup_{n\rightarrow\infty}\left\lVert\frac{1}{\delta_n}R_{\delta_n}^\frac{1}{p}(|x-y|)(v(x)-v(y))\right\rVert_{L^p(\Omega\times\Omega)}\\
				&\quad	=\limsup_{n\rightarrow\infty}\left\lVert\frac{1}{\delta_n}R_{\delta_n}^\frac{1}{p}(|x-y|)\nabla v(x)\cdot(x-y)\right\rVert_{L^p(\Omega\times\Omega)}\\
				&\quad	=\limsup_{n\rightarrow\infty}\left(\int_{\Omega}\int_{\Omega}R_{\delta_n}(|x-y|)|\nabla v(x)\cdot\frac{x-y}{\delta_n}|^pdxdy\right)^\frac{1}{p}		<\infty.
				\end{aligned}
			\end{equation*}

So the assumption $\liminf_{n\rightarrow\infty}\tilde{F}_n^v(u_n)<\infty$ also yields that
    \begin{equation*}
        \begin{aligned}
      &	\liminf_{n\rightarrow\infty}\frac{1}{\delta_n}\left(\int_\Omega\int_\Omega R_{\delta_n}(|x-y|)|u_n(x)-u_n(y)|^pdxdy\right)^\frac{1}{p}
 \\
 & \leq\left\lVert\frac{1}{\delta_n}R_{\delta_n}^\frac{1}{p}(|x-y|)(u(x)-u(y)+v(x)-v(y))\right\rVert_{L^p(\Omega\times\Omega)} \\
 & \quad + 
 \left\lVert\frac{1}{\delta_n}R_{\delta_n}^\frac{1}{p}(|x-y|)(v(x)-v(y))\right\rVert_{L^p(\Omega\times\Omega)}
 <\infty. 
   \end{aligned}
    \end{equation*}
   
			Therefore, with \cref{lem:L2-estimation}, we get the conclusion that the $W^{1,p}$ norm of $\{\tilde{u}_n\}$ is uniformly bounded
			\[\sup_n \lVert\tilde{u}_n\rVert_{W^{1,p}(\Omega)}<\infty.\]
			Finally, the transform form $u_n$ to $\tilde{u}_n$ preserves the $L^p$ convergence due to \cref{lem:fou}:
			\begin{equation*}
				\begin{aligned}
					\lVert\tilde{u}_n-u\rVert_{L^p(\Omega)}&=\left\lVert\frac{1}{\omega_{\delta_n}(x)}\int_\Omega K_{\delta_n}(|x-y|)(u_n(y)-u(x))dy\right\rVert_{L^p(\Omega)}\\
					&\leq C_3\left\lVert\int_\Omega K_{\delta_n}(|x-y|)(u_n(y)-u(x))dy\right\rVert_{L^p(\Omega)}\\
					&\leq C_3\bigg(\left\lVert\int_\Omega K_{\delta_n}(|x-y|)(u_n(y)-u_n(x))dy\right\rVert_{L^p(\Omega)}\\
					&+\left\lVert\int_\Omega K_{\delta_n}(|x-y|)(u_n(x)-u(x))dy\right\rVert_{L^p(\Omega)}\bigg)\\
					&\leq C_3\left\lVert\int_\Omega K_{\delta_n}(|x-y|)(u_n(y)-u_n(x))dy\right\rVert_{L^p(\Omega)}\\
					&+C_4\lVert u_n-u\rVert_{L^p(\Omega)}\longrightarrow 0, \quad (\text{as } \, n\rightarrow\infty).
				\end{aligned}
			\end{equation*}
			Hence, with \cref{lem:trace}, $\lVert Tu\rVert_{L^p(\partial\Omega)}=0$.
		\end{proof}
  
	For the limsup inequality, the following lemma is useful. It provides the connection between the nonlocal  
  Dirichlet energy and the local version defined on $W^{1,p}(\Omega)$, The proof can be done using the Taylor expansion, see proofs, in for example, \cite{BBM01}.
  
        \begin{lemma}
            \label{lem:taylor}
            Suppose that $\Omega$ is a Lipschitz bounded domain in $\mathbb{R}^d$. $R$ is a  kernel satisfying (K1)-(K3). Then, for all $u\in 
            W^{1,p}(\Omega)
            $, $\delta>0$,
            \begin{equation*}
                \begin{aligned}
                \frac{1}{\delta^p}\int_\Omega\int_\Omega R_{\delta}(|x-y|)|u(x)-u(y)|^pdxdy\leq\sigma_R\lVert\nabla u\rVert_{L^p(\Omega)}^p.
                \end{aligned}
            \end{equation*}
            
        \end{lemma}
   
  
		\begin{lemma}[the limsup inequality]
			\label{lem:limsup}
			Suppose that $\Omega$ is a Lipschitz bounded domain in $\mathbb{R}^d$. $K,R$ are two kernel functions satisfying (K1)-(K3). $\delta_n\rightarrow 0$ as $n\rightarrow\infty$. Then for any $u$ in $L^p(\Omega)$, there exists a sequence $\{u_n\}$ converging to $u$ in $L^p(\Omega)$ and
			\[\limsup_{n\rightarrow\infty}\tilde{F}_n^v(u_n)\leq \sigma_RF^v(u),\]
			where $\tilde{F}_n^v,F^v$ are defined as \eqref{eq:dis-fun3},\eqref{eq:con-fun2}.
		\end{lemma}
		\begin{proof}
			We can suppose that $F^v(u)<\infty$, which yields $u\in W^{1,p}$ and $Tu\equiv 0$. With the density of $C^\infty$ functions in $W^{1,p}$ space, we can choose $\{u_n\}$
to be a sequence of smooth functions and converges to $u$ with respect to the $W^{1,p}(\Omega)$ norm. We additionally require $u_n\equiv 0$ in $\Omega_{\delta_n}^c$, where $\Omega_{\delta_n}\coloneqq\{x\in\Omega\big| \rm{dist}(x,\Omega^c)\geq r_K \delta_n\}$. This condition can be attained by mutiplying  a smooth approximation of the indicator function $\mathbf{1}_{\Omega_{\delta_n}}$. 	

			For such $\{u_n\}$, it is obvious that
			\[\left(\int_{\partial\Omega}\left|\frac{1}{\delta_n}\int_\Omega K_{\delta_n}(|x-y|)u_n(y)dy\right|^pdx\right)^\frac{1}{p}=0.\]
			To simplify the notation, we denote $u_n+v,u+v$ by $u_n^v,u^v$. With \cref{lem:taylor},
		\begin{equation*}
		\begin{aligned}
				\frac{1}{\delta_n^p}\int_\Omega\int_\Omega R_{\delta_n}(|x-y|)|u_n^v(x)-u_n^v(y)|^pdxdy		\leq\sigma_R\lVert\nabla u_n^v(z)\rVert_{L^p(\Omega)}^p, \;  \forall n>o.
			\end{aligned}
		\end{equation*}
			As a result, with $u_n^v\rightarrow u^v$ in $W^{1,p}(\Omega)$,
			\begin{equation*}
				\begin{aligned}			\limsup_{n\rightarrow\infty}\tilde{F}_n^v(u_n)
	& =\limsup_{n\rightarrow\infty}\frac{1}{\delta_n^p}\int_\Omega\int_\Omega R_{\delta_n}(|x-y|)(u_n^v(x)-u_n^v(y)|^pdxdy\\
&\leq\sigma_R\limsup_{n\rightarrow\infty}\lVert\nabla u_n^v\rVert_{L^p(\Omega)}^p					=\sigma_R\lVert\nabla u^v\rVert_{L^p(\Omega)}^p	=\sigma_RF^v(u).			\end{aligned}
		\end{equation*}
		\end{proof}
  
	Finally, with	 \cref{lem:approx,lem:liminf,lem:limsup}, the proof  of \cref{thm:gamma-con} is complete. 
		\subsection{Nonlocal model for eigenvalue problem}
\label{subsec:proof2}
		Recall that $F_{e,n}^k$, $\tilde{F}_{e,n}^k$ and $F_e^k$ defined as \eqref{eq:eigen-dis-fun},\eqref{eq:normal-dis-fun} and \eqref{eq:eigen-cons-func} actually equal to $F_n$ or $F$ if $u$ satisfies the constraints. With the $\varGamma$-convergence of $F_n$ to $F$ proved in \cref{thm:gamma-con}, we only need to complete the remaining proof.
		\begin{proof}[Proof of \cref{thm:eigenvalue}] we apply the induction with respect to $k$.
  
			$Unnormalized\ case,\ liminf\ inequality$: 
   \gwy{Without loss of generality, we can assume that $\liminf_{n\rightarrow\infty}F_{e,n}^k(u_n)<\infty$.} Then there exists a subsequence $\{n_l\}$ such that $\lVert u_{n_l}\rVert_{L^2(\Omega)}=1$ and $\ u_{n_l}\perp V_{k-1}^n,\ \forall l.$ Note that $u_{n_l}\rightarrow u$ in $L^2(\Omega)$, so we have $\lVert u\rVert_{L^2(\Omega)}=1$. For $k=1$, $u\perp V_0$ is trivial since $V_0=\emptyset$. For $k>1$, suppose that $F_{e,n}^l$ $\varGamma-$converge to $F_e^l$ for all $l<k$. Then with the compactness result in \cref{sec:compact} and the property of $\varGamma-$convergence($\cref{lem:minimizer}$), we have $\mu_l^n\rightarrow\mu_l$ in $L^2$ norm for all $l<k$. Hence,
   \[\langle u,\mu_l\rangle=\lim_{n\rightarrow\infty}\langle u_n,\mu_l^n\rangle=0\]
   for all $l<k$, which means that $u\perp V_{k-1}$ and 		\begin{equation*}
		\begin{aligned}
\liminf_{n\rightarrow\infty}F_{e,n}^k(u_n)&=\liminf_{n\rightarrow\infty}F_n(u_n)
				\leq \sigma_RF(u)
			=\sigma_RF_e^k(u).	\end{aligned}
		\end{equation*}
  
		$Unnormalized\ case,\ limsup\ inequality$: Supposing that $F(u)<\infty$, from\\ \cref{lem:limsup} with $v\equiv 0$, we get a sequence $\{u_n\}$ converging to $u$ in $L^2$ norm. Meanwhile, $u_n$ is smooth and supported in $\Omega_{\delta_n}$. For $k=1$, replacing $u_n$ with $\frac{u_n}{\lVert u_n\rVert_{L^2}}$ and using the same progress in \cref{lem:limsup}, we can prove the limsup inequality in unnormalized case. For $k>1$, we consider a smooth approximation $\nu_m^n$ for a basis $\mu_m^n$ of $V_{k-1}^n$. $\nu_m^n$ is also supported in $\Omega_{\delta_n}$ and $\lVert\mu_m^n-\nu_m^n\rVert_{L^2}\rightarrow 0$. Let
			\[\tilde{u}_n=u_n-\sum_{m=1}^{k-1}\alpha_m^n\nu_m^n,\;
		\text{ where }\;
\alpha^n\coloneqq\left(\begin{array}{c}
			\alpha_1^n\\
				\vdots\\
				\alpha_{k-1}^n
			\end{array}\right)=G_n^{-1}b_n,\]
			\[G_n=\left(\begin{array}{ccc}
				\langle\mu_1^n,\nu_1^n\rangle&\dots&\langle\mu_1^n,\nu_{k-1}^n\rangle\\			\vdots&\ddots&\vdots\\			\langle\mu_{k-1}^n,\nu_1^n\rangle&\dots&\langle\mu_{k-1}^n,\nu_{k-1}^n\rangle\\			\end{array}\right),\ b_n=\left(\begin{array}{c}		\langle\mu_1^n,u_n\rangle\\
				\vdots\\
				\langle\mu_{k-1}^n,u_n\rangle		\end{array}\right)\]
			and $\displaystyle \langle u,v\rangle=\int_\Omega uvdx$ is the inner product of $L^2(\Omega)$. $G_n$ is invertible
   for large enough $n$ because $G_n\rightarrow I$ as $n\rightarrow\infty$. It can be verified that with such modification, $\tilde{u}_n$ is still smooth, supported in $\Omega_{\delta_n}$ and becomes perpendicular to $V_{k-1}^n$. Furthermore, with the induction assumption, $b_n\rightarrow\mathbf{0}$ as $n\rightarrow\infty$ so the $L^2$ convergence is also retained. Finally, by replacing $\tilde{u}_n$ with $\frac{\tilde{u}_n}{\lVert\tilde{u}_n\rVert_{L^2}}$ and the conclusion can be proved in the same way as in \cref{lem:limsup}.
   
   $Normalized\ case$: The proof in normalized case is similar to the one in unnormalized case with the fact that
   \[\lim_{n\rightarrow\infty}|\langle u,v_n\rangle_n-\langle u,v\rangle|=0\]
   for all $u\in L^2(\Omega)$ and $v_n\rightarrow v$ in $L^2$ norm. Actually, by extending the domain of $v$, $v_n$ and $u$ to $\mathbb{R}^d$ with zero-value outside $\Omega$. Then, with \cref{lem:fou},		\begin{equation*}
		\begin{aligned}
			&|\langle u,v_n\rangle_n-\langle u,v\rangle|\\	=&\left|\int_{\Omega}\int_{\Omega}W_{\delta_n}(|x-y|)u(x)v_n(y)dxdy-\int_{\Omega}u(x)v(x)dx\right|\\
=&\left|\int_{\mathbb{R}^d}\int_{\mathbb{R}^d}W_{\delta_n}(|x-y|)u(x)v_n(y)dxdy-\int_{\mathbb{R}^d}u(x)v(x)dx\right|\\
\leq&\int_{\mathbb{R}^d}\int_{\mathbb{R}^d}W_{\delta_n}(|x-y|)|v(x)-v_n(y)|dy|u(x)|dx\\	\leq&\left\lVert\int_{\Omega}W_{\delta_n}(|x-y|)|v(x)-v_n(y)|dy\right\rVert_{L^2(\mathbb{R}^d)}\lVert u\rVert_{L^2(\mathbb{R}^d)}\to 0, \; \text{as } n\rightarrow\infty.
		\end{aligned}
		\end{equation*}
The proof is complete.
		\end{proof}

		\subsection{Nonlocal model with general boudary term}
		\label{subsec:proof3}
		The proof of \cref{thm:gamma-con} allows us to extract the requirements for the boundary terms in the nonlocal model. To fulfill the liminf inequality, we need $\{\tilde{u}_n\}$ satisfy the condition of \cref{lem:trace}. Among them, the part involved in the boundary term is that $\lVert\tilde{u}_n\rVert_{L^2(\partial\Omega)}$ should tend to $0$ as $n\rightarrow\infty$. This is the reason why we require coercivity presented in \autoref{thm:gamma-con-general}. Regarding the limsup inequality, the construction in \cref{lem:limsup} can be employed to maintain the boundary term at $0$ as long as the kernel has a gradually shrinking compact support.
		\begin{proof}[Proof of \cref{thm:gamma-con-general}]
{\em The liminf inequality}. From \cref{lem:gamma-con}, we only need to demonstrate that $Tu\equiv 0$ on $\partial\Omega$ with $u_n\rightarrow u$ in $L^2$ norm and
			\begin{equation*}
				\begin{aligned}
					\liminf_{n\rightarrow\infty} F_n(u_n)<\infty.
				\end{aligned}
			\end{equation*}
			The coercivity of boundary term $E(u,0)$ derives that there exists a kernel $K$ satisfying (K1)-(K3) such that
			\begin{equation*}
				\begin{aligned}
					\lim_{n\rightarrow\infty}\lVert\tilde{u}_n\rVert_{L^2(\partial\Omega)}=0
				\end{aligned}
			\end{equation*}
			up to a subsequence. $\tilde{u}_n\rightarrow u$ in $L^2$ norm with \cref{lem:fou}. And \cref{lem:L2-estimation} gives the uniform boundedness of $\lVert\tilde{u}_n\rVert_{H^1(\Omega)}$. Therefore, using \cref{lem:trace}, the proof of liminf inequality is completed.\\
	
 {\em The limsup inequality}. With the similar method in \cref{lem:limsup}, we can suppose $F(u)<\infty$ and select $\{u_n\}$ as a sequence of $C^\infty$ approximations converging to $u$ with respect to $H^1$ norm. Moreover we require $u_n\equiv 0$ in $\Omega_{c_n}^c$, where $\Omega_{c_n}\coloneqq\{x\in\Omega\big|\text{dist}(x,\Omega^c)\geq c_n\}$. Such property of $\{u_n\}$ leads to the vanishing of boudary term $E_n(u,0)$ because the intersection of the support sets of the kernel $\rho_{\delta_n,x}$ and $u_n$ is empty. As for the first term, using \cref{lem:taylor}, we have
			\begin{equation*}
				\begin{aligned}
					\frac{1}{\delta_n^2}\int_\Omega\int_\Omega R_{\delta_n}(|x-y|)|u_n(x)-u_n(y)|^2dxdy\leq\sigma_R\lVert\nabla u_n(z)\rVert_{L^2(\Omega)}^2
				\end{aligned}
			\end{equation*}
   for all $n$.
			Hence,
			\begin{equation*}
				\begin{aligned}
					&\limsup_{n\rightarrow\infty}F_n(u_n)\\
					&\quad				=\limsup_{n\rightarrow\infty}\frac{1}{\delta_n^2}\int_\Omega\int_\Omega R_{\delta_n}(|x-y|)|u_n(x)-u_n(y)|^2dxdy\\
&\quad					\leq\limsup_{n\rightarrow\infty}\sigma_R\lVert\nabla u_n\rVert_{L^2(\Omega)}^2	=\sigma_R\lVert\nabla u\rVert_{L^2(\Omega)}^2
					=F(u).
				\end{aligned}
			\end{equation*}
		\end{proof}
		
		The construction in \cref{thm:gamma-con} is not the only one that fulfills our sufficient conditions. Another 
 example is, as we have mentioned in \cref{remark1}, selecting $\rho_{\delta_n,x}(y,z)=\frac{1}{\delta_n^2} K_{\delta_n}(|x-y|)\delta(|x-y|)$. In this case, our nonlocal functional is formulated as
		\begin{equation}
			\label{eq:dis-fun-cor}
			\begin{aligned}
				F_n(u)&=\frac{1}{\delta_n^2}\int_\Omega\int_\Omega R_{\delta_n}(|x-y|)|u(x)-u(y)|^2dxdy\\
				&+\frac{1}{\delta_n^2}\int_{\partial\Omega}\int_\Omega K_{\delta_n}(|x-y|)u(y)^2dydx.\\
			\end{aligned}
		\end{equation}
		and we have the following corollary due to \cref{thm:gamma-con-general}:
		\begin{corollary}
			Suppose that $\Omega$ is a Lipschitz bounded domain in $\mathbb{R}^d$. $R$ is a kernel satisfying (K1)-(K3). $K$ is a non-negative compactly supported kernel with a uniform positive lower bound in a neighborhood of the origin. Then, we have
			\[F_n\stackrel{\varGamma}{\longrightarrow}\sigma_RF\quad\text{in }L^2(\Omega),\]
			where $F_n,F$ are defined as \eqref{eq:dis-fun-cor},\eqref{eq:con-fun-special}.
		\end{corollary}
		\begin{proof}
			$\rho_{\delta_n,x}$ has a shrinking compact support since 
			\gwy{\begin{equation*}
				\begin{aligned}
					\rho_{\delta_n,x}(y,z)=\frac{1}{\delta_n^2} K_{\delta_n}(|x-y|)\delta(|y-z|)
				\end{aligned}
			\end{equation*}}
			and the kernel $K$ is compactly supported.
			
			Note that $K$ is non-negative and there exists constants $c_1,\ c_2>0$ such that $K(s)<c_1$ for all $s\in[0,c_2]$. Define a new kernel $\hat{K}:[0,\infty)\rightarrow[0,\infty)$,
			\begin{equation*}
				\begin{aligned}
					\hat{K}(s)\coloneqq\left\{\begin{array}{cc}
						\frac{c_1}{c_2^2}(s-c_2)^2&\text{if }s\leq c_2,\\
						0&\text{if }s> c_2.
					\end{array}\right.
				\end{aligned}
			\end{equation*}
			It can be easily verified that $\hat{K}$ satisfies (K1)-(K3) and $K(s)\geq\hat{K}(s)$ for all $s\geq 0$. Hence, 
  with $\hat{K}_\delta$ being the rescaled kernel defined by $\hat{K}$, we have
\begin{equation*}
				\begin{aligned}
					E_n(u,0)&=\frac{1}{\delta_n^2}\int_{\partial\Omega}\int_\Omega K_{\delta_n}(|x-y|)u(y)^2dydx\\
					&\geq\frac{1}{\delta_n^2}\int_{\partial\Omega}\int_\Omega \hat{K}_{\delta_n}(|x-y|)u(y)^2dydx\\
					&\geq\frac{C}{\delta_n^2}\int_{\partial\Omega}\int_\Omega \hat{K}_{\delta_n}(|x-y|)dy\int_\Omega \hat{K}_{\delta_n}(|x-y|)u(y)^2dydx\\
					&\geq \frac{C}{\delta_n^2}\int_{\partial\Omega}\left(\int_\Omega \hat{K}_{\delta_n}(|x-y|)
     u(y)dy\right)^2dx
				\end{aligned}
			\end{equation*}
			which means that the functional $E_n$ associated with $\rho_{\delta_n,x}$ is coercive as defined in \cref{thm:gamma-con-general} for $C_n=\frac{\delta_n^2}{C}$ and
   \[\tilde{u}_n=\frac{1}{w_{\delta_n}(x)}\int_\Omega \hat{K}_{\delta_n}(|x-y|)
     u(y)dy.\]
     Therefore, with \cref{thm:gamma-con-general}, $\{F_n\}$ $\varGamma$-converges to $F$.
		\end{proof}

	\section{Compactness}
	In this section, we demonstrate that any minimizing sequence of the nonlocal functionals $\{F_n\}$ defined as \eqref{eq:dis-fun} or \eqref{eq:dis-fun-general}, $\{F_{e,n}^k\}$ defined as \eqref{eq:eigen-dis-fun} and $\{\tilde{F}_{e,n}^k\}$ defined as \eqref{eq:normal-dis-fun} is relatively compact in $L^p(\Omega)$. 
In the literature,  such kind of compactness result has been studied for nonlocal functions using the techniques developed in \cite{BBM01,ponce2004new}. For smooth kernels, one can also directly work the compactness of the mollified sequences, see  for example \cite{scott2023nonlocal}. In our case, this corresponds to the use of mollification given by
 \eqref{eq:tilde-u}.
 Note that if $\{u_n\}$ is a minimizing sequence of $\{F_n\}$, then $\sup_n\{F_n(u_n)\}<\infty$. Hence it is sufficient to show that $\{u_n\}$ is relatively compact if one of the following three conditions holds: $\sup_n\{F_n(u_n)\}<\infty$, $\sup_n\{F_{e,n}^k(u_n)\}<\infty$ and $\sup_n\{\tilde{F}_{e,n}^k(u_n)\}<\infty$. Recall that $F_{e,n}^k$ and $\tilde{F}_{e,n}^k$ are $F_n$ with additional constraints added, which means that $F_{e,n}^k(u),\ \tilde{F}_{e,n}^k(u)\geq F_n(u)$ for all $u\in L^2(\Omega)$. As for $F_n$, the relative compactness can be derived mainly by the interior term of the nonlocal functional defined as follow:
	\label{sec:compact}
	\begin{equation}
		\label{eq:func-interior}
		\begin{aligned}
			F_n^i=\frac{1}{\delta_n}\left(\int_\Omega\int_\Omega R_{\delta_n}(|x-y|)|u(x)-u(y)|^pdxdy\right)^\frac{1}{p}
		\end{aligned}
	\end{equation}
	\begin{lemma}
	    \label{lem:compactness}
		Suppose that $\Omega$ is a Lipschitz bounded domain in $\mathbb{R}^d$. $1<p<\infty$. $R$ is a kernel satisfying (K1)-(K3). $\{\delta_n\}$ is a sequence of positive constants tending to $0$ as $n\rightarrow\infty$. $\{u_n\}$ is a bounded sequence in $L^p(\Omega)$ and satisfies
		\[\sup_nF_n^i(u_n)<\infty\]
		where $F_n^i$ is defined as \eqref{eq:func-interior}. Then $\{u_n\}$ is a relatively compact sequence in $L^p(\Omega)$.
	\end{lemma}
	\begin{proof}
		Recalling $\tilde{u}_n$ defined as \eqref{eq:tilde-u}, we have shown that $\lVert\tilde{u}_n-u_n\rVert_{L^p(\Omega)}\rightarrow 0$ as $n\rightarrow\infty$ in \cref{lem:fou}. Hence, $\{\tilde{u}_n\}$ is also bounded in $L^p(\Omega)$. With the condition $\sup_nF_n^i(u_n)<\infty$ and \cref{lem:L2-estimation}, $\{\tilde{u}_n\}$ is actually bounded in $W^{1,p}(\Omega)$. Hence, $\{\tilde{u}_n\}$ is relatively compact in $L^p(\Omega)$ due to the Rellich-Kondrachov theorem (see, for example,  theorem 6.3 in \cite{adams2003sobolev}). So is $\{u_n\}$, because $\{\tilde{u}_n\}$ and $\{u_n\}$ are asymptotically approximated in $L^p(\Omega)$. 
	\end{proof}
	
	Note that $\{u_n\}$ is required to be bounded in \cref{lem:compactness}. We claim that this requirement can be deduced by the boundedness of $\{F_n(u_n)\}$ and no additional assumptions are needed. Recall the Poincar\'e inequality (see, for example, in section 5.8.1 of \cite{Evans}). Let $\Omega$ be a Lipschitz bounded domain in $\mathbb{R}^d$, $1\leq p\leq\infty$. Then there exist a constant $C$, depending only on $d,p$ and $\Omega$, such that
	\begin{equation}
	    \label{eq:poincare}
	    \lVert u-(u)_\Omega\rVert_{L^p(\Omega)}\leq C\lVert \nabla u\rVert_{L^p(\Omega)}
	\end{equation}
	for each function $u\in W^{1,p}(\Omega)$ and $(u)_\Omega\coloneqq\frac{1}{|\Omega|}\int_\Omega u dx$ is the average of u over $\Omega$. In the nonloacal model $F_n$ defined as \eqref{eq:dis-fun}, the interior term and the boundary term can be consider as the approximation of $\lVert\nabla u\rVert_{L^p(\Omega)}$ and $\lVert u\rVert_{L^p(\partial\Omega)}^p$ respectively. Hence, with the aid of Poincar\'e inequality, we can establish its nonlocal counterpart as the following lemma.
    \begin{lemma}
        \label{lem:poincare}
        Suppose that $\Omega$ is a Lipschitz bounded domain in $\mathbb{R}^d$. $1<p<\infty$. $K,R$ are two kernel functions satisfying (K1)-(K3). $\{\delta_n\}$ is a sequence of positive constants tending to 0 as $n\rightarrow\infty$. $\{u_n\}\subset L^p(\Omega)$ is a sequence satisfying $\sup_n F_n(u
        _n)<\infty$ where $F_n$ is defined as \eqref{eq:dis-fun}. Then
        \gwy{\[\sup_n \lVert u_n\rVert_{L^p(\Omega)}<\infty.\]}
    \end{lemma}
	\begin{proof}
	    \gwy{Define \begin{equation*}
			 \tilde{u}_n (x)\coloneqq \frac{1}{\omega_{\delta_n}(x)}\int_\Omega \hat{K}_{\delta_n}(|x-y|)u_n(y)dy, \; \text{ and }\;\omega_{\delta_n}(x)\coloneqq \int_\Omega \hat{K}_{\delta_n}(|x-y|)dy.
		  \end{equation*}}  
     With \cref{lem:L2-estimation} and the condition that $\sup_n F_n(u_n)<\infty$, 
     we have
   $\sup_n 
     \lVert \nabla \tilde{u}_n\rVert_{L^p(\Omega)}$\\
     $<\infty$
     and $\sup_n 
     \lVert \tilde{u}_n\rVert_{L^p(\partial\Omega)}<\infty$.
     By the classical Poincar\'e inequality,
    there is a constant
$C_1(p,d,\Omega)>0$ such that
    \begin{equation*}
	            \lVert u\rVert_{L^p(\Omega)}\leq  C_1(p,d,\Omega)(\lVert \nabla u\rVert_{L^p(\Omega)}+ \lVert u\rVert_{L^p(\partial\Omega)}
             ),\;
             \forall u\in W^{1,p}(\Omega).
	    \end{equation*}
    So 
     $\sup_n 
     \lVert \tilde{u}_n \rVert_{L^p(\Omega)} <\infty$. Meanwhile, as $n\to \infty$, we have $
       \lVert \tilde{u}_n -u_n \rVert_{L^p(\Omega)}\to 0$  by the proof  of \cref{lem:fou},. Hence, we have the desired uniform bound of 
     $\{\lVert u_n \rVert_{L^p(\Omega)}\}$.
	\end{proof}

	For $F_n$ defined as \eqref{eq:dis-fun-general}, we can obtain the boundedness of $\{u_n\}$ with the same method in \cref{lem:poincare} under the coercivity assumption. With the compactness result above, \cref{thm:gamma-con-minimizers,thm:eigenvalue-minimizers,thm:gamma-con-general-minimizers} can be demonstrated with the $\varGamma$-convergence \cref{thm:gamma-con,thm:eigenvalue,thm:gamma-con-general} and its property \cref{lem:minimizer}.
	
	\section{Conclusion}
	\label{sec:conclusion}
	In this paper, we propose a penalty formulation for some variational nonlocal Dirichlet problems.
Sufficient conditions for the boundary terms of these models to ensure $\varGamma$-convergence are presented. Based on this work, there are several aspects for future research. Firstly, the coercivity proposed in \cref{thm:gamma-con-general}
 may be difficult to verify in certain situations. Alternative conditions that are 
more intuitive, albeit stronger, may be explored. The 
conditions studied in this work
are only sufficient, and it would be valuable to investigate necessary conditions as well. Furthermore, it is also worthy consideration to extend the study here to models such as biharmonic equations, Stokes systems, and other linear and nonlinear equations of broad interests.

	\bibliographystyle{siamplain} 
	\bibliography{references}

\end{document}